\documentclass{birkart}
\usepackage{amsmath}
\usepackage{amssymb}
\usepackage{amsfonts}
\usepackage{graphicx}
\usepackage{texdraw}
\usepackage{graphpap}

\input txdtools

 \newtheorem{thm}{Theorem}[section]
 
 \newtheorem{lem}[thm]{Lemma}
 \newtheorem{prop}[thm]{Proposition}
 \theoremstyle{definition}

 \theoremstyle{remark}
 
 \newtheorem{rem}[thm]{Remark}

\numberwithin{equation}{section}
\numberwithin{figure}{section}

\newcommand{\beur}{{\mathbf B}}
\newcommand{\proj}{{\mathbf P}}

\renewcommand{\d}{{\boldsymbol\partial}}
\newcommand{\dbar}{{\bar{\boldsymbol\partial}}}
\newcommand{\bDelta}{{\boldsymbol\Delta}}
\newcommand{\e}{\mathrm e}

\newcommand{\C}{{\mathbb C}}
\newcommand{\D}{{\mathbb D}}

\newcommand{\R}{{\mathbb R}}

\newcommand{\id}{\operatorname{\text{\bf I}}}
\newcommand{\sgn}{\operatorname{\text{sgn}}}

\newcommand{\calF}{{\mathcal F}}

\newcommand{\re}{\operatorname{Re}}
\newcommand{\im}{\operatorname{Im}}

\newcommand{\Qop}{{\mathbf Q}}
\newcommand{\Cop}{{\mathbf C}}
\newcommand{\Dop}{{\mathbf D}}
\newcommand{\Eop}{{\mathbf E}}

\newcommand{\Mop}{{\mathbf M}}
\newcommand{\Top}{{\mathbf T}}
\newcommand{\Hplane}{{\mathbb H}}

\newcommand{\diff}{{\mathrm d}}
\newcommand{\imag}{{\mathrm i}}

\begin{document}
%
\title[The Beurling operator for the hyperbolic plane]
{The Beurling operator for the hyperbolic plane}

\author[Hedenmalm]
{H\aa{}kan Hedenmalm}

\address{Hedenmalm: Department of Mathematics\\
The Royal Institute of Technology\\
S -- 100 44 Stockholm\\
SWEDEN}

\email{haakanh@math.kth.se}

\thanks{Research partially supported by the G\"oran Gustafsson Foundation and
by the Swedish Science Council (Vetenskapsr\aa{}det).}








\maketitle

\addtolength{\textheight}{2.2cm}







\noindent{\bf Abstract. } We find a Beurling operator for the hyperbolic 
plane, and obtain an $L^2$ norm identity for it, as well as two-sided 
$L^p$ estimates.

\section{Introduction and statement of main results}

\noindent\bf Outline of the paper. \rm We first mention the classical Cauchy
and Beurling operators $\Cop$ and $\beur$ in the setting of the plane. We
then introduce their hyperbolic plane analogues 
$\Cop^\uparrow,\Cop^\downarrow$ and $\beur^\uparrow,\beur^\downarrow$.
For instance, the hyperbolic Cauchy operator $\Cop^\downarrow$ finds the 
$L^2$-minimal solution to the $\dbar$-problem.
The mapping properties of $\Cop^\uparrow,\Cop^\downarrow$ rely on the 
well-known Hardy inequality for the upper half plane. We find a sharp two-sided
estimate for the the norm of $\beur^\downarrow[f]$ in weighted $L^p$-spaces,
which is analogous to the well-known two-sided estimate for $\beur[f]$. 
In the Hilbert space case $p=2$, the estimate becomes a norm isometry.
The way the proof is set up, we need some results of Liouville type for
the hyperbolic plane. 
In the final section, we try to explain the assertion of the main theorem in 
geometric terms. 
\medskip

\noindent\bf The Beurling transform. \rm
The Beurling transform (or operator) $\beur:\,L^2(\C)\to L^2(\C)$ is 
formally the operator $\beur={\d}\dbar^{-1}$. Here, we use the notation
$$\d_z=\frac{1}{2}\bigg(\frac{\partial}{\partial x}-\imag
\frac{\partial}{\partial y}\bigg),\quad \dbar_z=
\frac{1}{2}\bigg(\frac{\partial}{\partial x}+\imag
\frac{\partial}{\partial y}\bigg),$$
and put $\bDelta_z=\d_z\dbar_z$; as above, we frequently suppress 
the subscript $z$. This way of defining $\beur$ leaves some ambiguity, as
there are many possible ways to define $\dbar^{-1}$. The choice is to
use the Cauchy transform $\Cop$ for $\dbar^{-1}$,
$$\Cop[f](z)=\int_{\C}\frac{f(w)}{z-w}\,\diff A(w),\qquad z\in\C,$$
where 
$$\diff A(z)=\frac{\diff x\diff y}{\pi},\qquad z=x+\imag y,$$
is normalized area measure. Unfortunately, the integral defining $\Cop[f]$ is
not well-defined for all $f\in L^2(\C)$, but if $f$ is compactly supported,
there is no problem. Differentiating the Cauchy transform, we get
$$\beur[f](z)=-\text{pv}\int_{\C}\frac{f(w)}{(z-w)^2}\,\diff A(w),
\qquad z\in\C,$$
where ``pv'' stands for {\em principal value}. It is easy to show, using
Fourier analysis or Green's formula, that $\beur$ acts isometrically on 
$L^2(\C)$:
\begin{equation}
\|\beur[f]\|_{L^2(\C)}^2=\|f\|_{L^2(\C)}^2=\int_\C|f|^2\diff A,
\label{eq-0}
\end{equation}
where the rightmost identity defines the norm in $L^2(\C)$. It is well-known
that $\beur$ acts boundedly on $L^p(\C)$ for $1<p<+\infty$; let $B(p)$ denote
its norm, that is, the best constant such that
$$\|\beur[f]\|_{L^p(\C)}\le B(p)\|f\|_{L^p(\C)},\qquad f\in L^p(\C),$$
holds. It is easy to show that there is an estimate from below as well:
\begin{equation}
\frac{1}{B(p)}\|f\|_{L^p(\C)}\le
\|\beur[f]\|_{L^p(\C)}\le B(p)\|f\|_{L^p(\C)},\qquad f\in L^p(\C).
\label{eq-beur1}
\end{equation}
A well-known conjecture due to Tadeusz Iwaniec (see \cite{Iwan}, \cite{BMS},
\cite{DV}, \cite{BJ}) claims that
$$B(p)=\max\bigg\{p-1,\frac{1}{p-1}\bigg\},\qquad 1<p<+\infty.$$
An easy duality argument shows that with $p'=p/(p-1)$ (dual exponent),
$$B(p)=B(p'),\qquad 1<p<+\infty.$$
There is a formulation of \eqref{eq-beur1} which does not use singular
integrals:
\begin{equation}
\frac{1}{B(p)}\|\dbar g\|_{L^p(\C)}\le
\|\d g\|_{L^p(\C)}\le B(p)\|\dbar g\|_{L^p(\C)},\qquad g\in C^\infty_0(\C),
\label{eq-beur2}
\end{equation}
where $C^\infty_0(\C)$ is the space of compactly supported test functions.
\medskip

\noindent\bf The hyperbolic plane. \rm 
Let $\Hplane$ denote the hyperbolic plane; we shall use the model
$$\Hplane=\langle\C_+,\diff s_{\Hplane}\rangle,$$
where 
$$\C_+=\big\{z\in \C:\,\im z>0\big\}$$
is the upper half plane, and
$$\diff s_{\Hplane}(z)=\frac{|\diff z|}{\im z}$$
is the Poincar\'e metric. The hyperbolic area element is given by
$$\diff A_{\Hplane}(z)=\frac{\diff A(z)}{(\im z)^2}.$$

\noindent\bf Function spaces. \rm
For $0<p<+\infty$ and real $q$, we introduce the space $L^p_q(\C_+)$
of (equivalence classes of) area-Lebesgue measurable functions subject to the
integrability condition
$$\|f\|^p_{L^p_q(\C_+)}=\int_{\C_+}|f(z)|^p(\im z)^q\diff A(z)=
\int_{\C_+}|f(z)|^p(\im z)^{q+2}\diff A_{\Hplane}(z)<+\infty.$$
It is a Banach space for $1\le p<+\infty$. We realize that $L^2_{\!-\!2}(\C_+)$
has the interpretation of $L^2(\Hplane)$, the $L^2$ space over the hyperbolic
plane. 
\medskip

\noindent\bf Some notation. \rm
We shall at times need conjugate symbol operators, as defined by
$$\bar\Top[f]=\text{conj}\,(\Top[\bar f]),$$
and we apply this notational convention to all the operators considered here.
Moreover, if $\calF$ is a collection of complex-valued functions, we write
$\text{conj}\,(\calF)$ for the collection of complex conjugates of the 
functions in $\calF$.
\medskip

\noindent \bf Hardy's inequality for the upper half plane. \rm
By Hardy's inequality for the upper half space,
\begin{equation}
\int_{\C_+}|f(z)|^p\frac{\diff A(z)}{(\im z)^p}\le 2^{p/2}(1-1/p)^{-p}
\int_{\C_+}
\big(|\d f(z)|^2+|\dbar f(z)|^2\big)^{p/2}\diff A(z),
\label{eq-hardy}
\end{equation}
for $f\in C^\infty_0(\C_+)$. 
The constant is sharp (see, e. g., \cite{Dav}, \cite{Maz}). If we use
that for $a,b\in\C$, 
$$(|a|^2+|b|^2)^{p/2}\le A(p)(|a|^p+|b|^p),\qquad A(p)=\max\{1,2^{-1+p/2}\},$$
we get
\begin{equation*}
\int_{\C_+}|f(z)|^p\frac{\diff A(z)}{(\im z)^p}\le 
2^{p/2}(1-1/p)^{-p}A(p)
\int_{\C_+}
\big(|\d f(z)|^p+|\dbar f(z)|^p\big)\diff A(z),
\end{equation*}
which in terms of norms reads
\begin{equation}
\|f\|^p_{L^p_{\!-\!p}(\C_+)}\le 
2^{p/2}(1-1/p)^{-p}A(p)\,
\big(\|\d f\|^p_{L^p(\C_+)}+\|\dbar f\|^p_{L^p(\C_+)}\big),\qquad f\in 
C^\infty_0(\C_+).
\label{eq-hardy-1'}
\end{equation}
Next, since by \eqref{eq-beur1},
$$\|\d f\|_{L^p(\C_+)}\le B(p)\,
\|\dbar f\|_{L^p(\C_+)},\qquad  f\in C^\infty_0(\C_+),$$
we obtain from \eqref{eq-hardy-1'} that
\begin{equation}
\|f\|^p_{L^p_{\!-\!p}(\C_+)}\le 2^{p/2}(1-1/p)^{-p}A(p)\,(1+B(p)^p)\,
\|\dbar f\|^p_{L^p(\C_+)},\qquad f\in C^\infty_0(\C_+).
\label{eq-hardy-2'}
\end{equation}
It is not obvious whether the constant appearing
on the right hand side of \eqref{eq-hardy-2'} is optimal for general $p$.
However, in case $p=2$, \eqref{eq-hardy-2'} reads
\begin{equation}
\|f\|_{L^2_{\!-\!2}(\C_+)}\le 4\,\|\dbar f\|_{L^2(\C_+)}, 
\qquad f\in C^\infty_0(\C_+),
\label{eq-hardy-3'}
\end{equation}
and the constant is sharp.
\medskip

\noindent\bf Cauchy operators associated with the upper half plane. \rm
For functions $f$ defined on $\C_+$, we introduce the Cauchy-type operators
$$\Cop^\downarrow[f](z)
=\int_{\C_+}\bigg(\frac{1}{z-w}-\frac{1}{z-\bar w}\bigg)f(w)\,
\diff A(w)=2\imag\int_{\C_+}\frac{f(w)\,\im w}{(z-w)(z-\bar w)}\,\diff A(w),
\qquad z\in\C_+,$$
and
$$\Cop^\uparrow[f](z)=
\int_{\C_+}\bigg(\frac{1}{z-w}-\frac{1}{\bar z-w}\bigg)f(w)\,
\diff A(w)=-2\imag\int_{\C_+}\frac{\im z\,f(w)}{(z-w)(\bar z-w)}\,\diff A(w),
\qquad z\in\C_+,$$
for all locally integrable functions $f$ for which the integrals
make sense (almost everywhere on $\C_+$). The operator $\Cop^\downarrow$
appears in the context of the unit disk in Subsection 4.8.3 of the book 
\cite{AIM} by Astala, Iwaniec, Martin.  
The identity
$$\bigg(\frac{1}{z-w}-\frac{1}{z-\bar w}\bigg)+\bigg(\frac{1}{\bar z-\bar w}
-\frac{1}{\bar z-w}\bigg)=\bigg(\frac{1}{z-w}-\frac{1}{\bar z-w}\bigg)
+\bigg(\frac{1}{\bar z-\bar w}-\frac{1}{z-\bar w}\bigg)$$
entails the operator identity
\begin{equation}
\Cop^\downarrow+\bar\Cop^\downarrow\equiv\Cop^\uparrow+\bar\Cop^\uparrow.
\label{eq-Copid}
\end{equation}
Moreover, with respect to the inner product of $L^2(\C_+)$, we have that
\begin{equation}
(\Cop^\downarrow)^*=-\bar\Cop^\uparrow,\quad
(\bar\Cop^\downarrow)^*=-\Cop^\uparrow,\quad
(\Cop^\uparrow)^*=-\bar\Cop^\downarrow,\quad
(\bar\Cop^\uparrow)^*=-\Cop^\downarrow.
\label{eq-Copidadj}
\end{equation}
To understand the action of 
$\Cop^{\uparrow}$, we note that 
\begin{multline}
F(z)-\Cop^\uparrow[\dbar F](z)=
-\int_{\C_+}\bigg(\frac{1}{z-w}-\frac{1}{\bar z-w}\bigg)\bar\partial 
F(w)\,\diff A(w)\\
=\int_{\C_+}\bar\partial_w\bigg\{
\bigg(\frac{1}{w-z}-\frac{1}{w-\bar z}\bigg)F(w)\bigg\}\diff A(w)
=\frac{1}{2\pi\imag}\int_{\R}\bigg(\frac{1}{w-z}-\frac{1}{w-\bar z}\bigg)
F(w)\,\diff w,\qquad z\in\C_+,
\label{eq-green}
\end{multline}
provided $F$ and $\dbar F$ are smooth and taper off relatively quickly to $0$
at infinity (the middle integral is to be interpreted in the sense of 
distributions theory). 
As a first application of \eqref{eq-green}, we find that
\begin{equation}
\Cop^\uparrow[\dbar f]=f,\qquad f\in C_0^\infty(\C_+).
\label{eq-1.1}
\end{equation}
In $L^2(\C_+)$, the closure of $\dbar C^\infty_0(\C_+)$ equals
 $L^2(\C_+)\ominus\text{conj}(A^2(\C_+))$ (this fact is known as Havin's 
lemma).
A second application of \eqref{eq-green} shows that
\begin{equation}
\Cop^\uparrow[g]=0,\qquad g\in \text{conj}(A^2(\C_+)),
\label{eq-zero}
\end{equation}
which means that we have determined the action of $\Cop^\uparrow$ on all of 
$L^2(\C_+)$. 
It now follows from \eqref{eq-1.1} and \eqref{eq-zero} combined 
with \eqref{eq-hardy-3'} that
\begin{equation}
\|\Cop^\uparrow[g]\|_{L^2_{\!-\!2}(\C_+)}\le 4\|g\|_{L^2(\C_+)},
\qquad g\in L^2(\C_+).
\label{eq-hardy-2}
\end{equation}
Expressed differently, the operator
\begin{equation}
\Cop^\uparrow:\,L^2(\C_+)\to L^2_{\!-\!2}(\C_+)
\label{eq-cauchy}
\end{equation}
is bounded and has norm $4$. A similar argument based on \eqref{eq-hardy-2'}
shows that 
$$\Cop^\uparrow:\,L^p(\C_+)\to L^p_{\!-\!p}(\C_+), \qquad 1<p<+\infty,$$
with a norm bound which depends on $p$. Let $C^\uparrow(p)$ be the norm of 
this operator. By \eqref{eq-Copidadj}, and the fact that with respect to the
inner product of $L^2(\C_+)$, the spaces $L^p_{\!-\!p}(\C_+)$ and 
$L^{p'}_{p'}(\C_+)$ are dual to one another (here, $p'=p/(p-1)$), we have
that
$$\Cop^\downarrow:\,L^p_p(\C_+)\to L^p(\C_+), \qquad 1<p<+\infty,$$
is bounded as well; we denote its norm by $C^\downarrow(p)$. 
The duality argument actually gives that
$$C^\downarrow(p)=C^\uparrow(p'),\qquad p'=p/(p-1).$$
As noted previously, for $p=2$, we have $C^\downarrow(2)=C^\uparrow(2)=4$.

We specialize for a moment to $p=2$ and look for an interpretation of 
the operator $\Cop^\downarrow$. 
By duality, the information on the null space of $\Cop^\uparrow$ supplied by
\eqref{eq-zero} leads to information on the range of $\Cop^\downarrow$:
$$\Cop^\downarrow:L^2_2(\C_+)\to L^2(\C_+)\ominus A^2(\C_+).$$
The operator $\Cop^\downarrow$ therefore furnishes the least norm 
solution to the $\dbar$-problem: $u=\Cop^\downarrow[f]$ has smallest 
norm in $L^2(\C_+)$ among all solutions to
$$\dbar u=f(z),\qquad z\in\C_+.$$

\medskip

\noindent\bf Double singularity Cauchy-type integral operators. \rm
We introduce the operators $\Dop^\uparrow$, $\Dop^\downarrow$, as given by
$$\Dop^\uparrow[g](z)=\int_{\C_+}\frac{g(w)}{(z-w)(\bar z-w)}\,\diff A(w),
\qquad z\in\C_+,$$
and
$$\Dop^\downarrow[g](z)=\int_{\C_+}\frac{g(w)}{(z-w)(z-\bar w)}\,\diff A(w),
\qquad z\in\C_+.$$
We readily check that
\begin{equation}
\Cop^\uparrow=-2\imag\,\Mop\Dop^\uparrow,\qquad
\Cop^\downarrow=2\imag\,\Dop^\downarrow\Mop,
\label{eq-4}
\end{equation}
where $\Mop[f](z)=(\im z)f(z)$. The analogous operators in the
setting of the unit disk $\D$ in place of $\C_+$ appeared recently in 
\cite{BarHed}.
The boundedness of the operators $\Cop^\uparrow$ and $\Cop^\downarrow$
in the corresponding contexts entails that 
$$\Dop^\uparrow:\,L^p(\C_+)\to L^p(\C_+),\qquad 
\Dop^\downarrow:\,L^p(\C_+)\to L^p(\C_+),$$ 
act boundedly for $1<p<+\infty$. Moreover, the norms of these operators 
may be expressed in terms of $C^\uparrow(p)$, $C^\downarrow(p)$.
In \cite{BarHed}, the operators $\Dop^\uparrow$ and $\Dop^\downarrow$ appeared
in the analysis of conformal maps. We may bring the analysis one step further
and consider, for a conformal mapping $\varphi:\C_+\to\Omega$, where 
$\Omega\subset\C$, the operators
$$\Dop^\uparrow_\varphi[g](z)
=\int_{\C_+}\frac{\varphi'(z)g(w)}{(\varphi(z)-\varphi(w))(\bar z-w)}
\,\diff A(w),\qquad z\in\C_+,$$
and 
$$\Dop^\downarrow_\varphi[g](z)=\int_{\C_+}
\frac{\varphi'(w)g(w)}{(\varphi(z)-\varphi(w))(z-\bar w)}\,\diff A(w),
\qquad z\in\C_+.$$
The instance $\varphi(z)=z$ gives us the operators $\Dop^\uparrow,
\Dop^\downarrow$ already mentioned. The more general operators 
$\Dop^\uparrow_\varphi,\Dop^\downarrow_\varphi$ deserve attention as well
(cf \cite{BarHed}).
\medskip

\noindent\bf The sum of two Cauchy-type operators. \rm The identity
\eqref{eq-Copid} suggests that we should study the operator
$$\Cop_{\text{sum}}=\Cop^\downarrow+\bar\Cop^\downarrow=
\Cop^\uparrow+\bar\Cop^\uparrow.$$
The mapping properties of $\Cop^\downarrow$ and $\Cop^\uparrow$ show that
$\Cop_{\text{sum}}$ maps boundedly ($1<p<+\infty$)
$$\Cop_{\text{sum}}:\,L^p_p(\C_+)\to L^p(\C_+),\quad
\Cop_{\text{sum}}:\,L^p(\C_+)\to L^p_{-p}(\C_+).$$
Interpolation theory allows us to combine these statements to get 
that $\Cop_{\text{sum}}$ maps boundedly ($1<p<+\infty$)
$$\Cop_{\text{sum}}:\,L^p_q(\C_+)\to L^p_{q-p}(\C_+),\qquad 0\le q\le p.$$
The calculation
$$\bigg(\frac{1}{z-w}-\frac{1}{z-\bar w}\bigg)+\bigg(\frac{1}{\bar z-\bar w}
-\frac{1}{\bar z-w}\bigg)=8\im z\,\im w\,\frac{\re(z-w)}{|(z-w)(z-\bar w)|^2}$$
shows that
$$\Cop_{\text{sum}}=8\,\Mop\Eop\Mop,$$
where $\Eop$ is the integral operator
$$\Eop[f](z)=\int_{\C_+}\frac{\re(z-w)}{|(z-w)(z-\bar w)|^2}\,f(w)\,\diff A(w),
\qquad z\in\C_+.$$
We read off from the mapping property of $\Cop_{\text{sum}}$ that $\Eop$ acts 
boundedly 
($1<p<+\infty$)
$$\Eop:\,L^p_{q-p}(\C_+)\to L^p_{q}(\C_+),\qquad
0\le q\le p.$$
For $p=2$ we even have that $\Eop:\,L^2_{q-2}(\C_+)\to L^2_{q}(\C_+)$ is 
norm contraction for $0\le q\le2$. 

\section{The Beurling transform for the hyperbolic plane}

\noindent\bf Beurling-type operators. \rm
We introduce the Beurling-type operators
$$\beur^\downarrow[f](z)=\d\Cop^\downarrow[f](z)=\text{pv}
\int_{\C_+}\bigg[\frac{1}{(z-\bar w)^2}-\frac{1}{(z-w)^2}\bigg]f(w)\,
\diff A(w),\qquad z\in\C_+,$$
and
$$\beur^\uparrow[f](z)=\text{pv}
\int_{\C_+}\bigg[\frac{1}{(\bar z-w)^2}-\frac{1}{(z-w)^2}\bigg]f(w)\,
\diff A(w),\qquad z\in\C_+,$$
for functions $f$ such that the above expressions make sense.
With respect to the inner product of $L^2(\C_+)$, we have the adjoint 
calculation formulas
$$(\beur^\downarrow)^*=\bar\beur^\uparrow,\quad 
(\bar\beur^\downarrow)^*=\beur^\uparrow,\quad (\beur^\uparrow)^*
=\bar\beur^\downarrow,\quad(\bar\beur^\downarrow)^*=\beur^\uparrow.$$
In analogy with \eqref{eq-Copid}, we have the operator identity
$$\beur^\downarrow+\bar\beur^\downarrow=\beur^\uparrow+\bar\beur^\downarrow.$$
If we extend $f$ to $\C$ by declaring it to vanish off $\C_+$, we have
$$\beur^\downarrow[f](z)=\beur[f](z)-\bar\beur[f](\bar z),\qquad z\in\C_+,$$
and
$$\beur^\uparrow[f](z)=\beur[f](z)-\beur[f](\bar z),\qquad z\in\C_+.$$
In view of \eqref{eq-beur1}, we see that $\beur^\downarrow$ and 
$\beur^\uparrow$ act boundedly on $L^p(\C_+)$ for $1<p<+\infty$, with
norm bound
\begin{equation}
\big\|\beur^\downarrow[f]\big\|_{L^p(\C_+)}\le2B(p)\,\|f\|_{L^p(\C_+)},
\qquad f\in L^p(\C_+).
\label{eq-BLp}
\end{equation}
The analogous bound holds for $\beur^\uparrow$ as well.

We shall obtain a more interesting result. For 
$1<p<+\infty$, $\beur^\downarrow$ acts boundedly on $L^p_{p}(\C_+)$, while 
$\beur^\uparrow$ acts boundedly on $L^p_{\!-\!p}(\C_+)$. It should be pointed
out here that the functions in $L^p_{p}(\C_+)$, extended to vanish on 
$\C\setminus\C_+$, need not be locally area-integrable near the real line,
and therefore it is not clear how, e. g., the operator $\beur$ could be
defined on $L^p_{p}(\C_+)$. {\em But $\beur^\downarrow$ is well-defined due
to the cancellation in the symbol}.

The spaces $L^p_{p}(\C_+)$ and $L^{p'}_{\!-\!p'}(\C_+)$ are dual to one 
another with respect to the inner product of $L^2(\C_+)$ (here $p'=p/(p-1)$
is the dual exponent). By interpolation theory, then, it follows from the
above that
$$\beur^\downarrow:\,L^p_q(\C_+)\to L^p_q(\C_+),\quad
\beur^\uparrow:\,L^{p}_{-q}(\C_+)\to L^{p}_{-q}(\C_+),$$
act boundedly for $0\le q\le p$. This range surely is not best possible, but an
understanding of when the Beurling operator $\beur$ is bounded in the weighted 
context (cf. \cite{PetVol}) combined with the approach presented here should
lead to the optimal range.  

As will be explained later on, the operators $\beur^\uparrow$ and 
$\beur^\downarrow$ are modifications of the Beurling operator $\beur$ to 
the setting of the hyperbolic plane. So it is natural to ask to what extent
\eqref{eq-beur1} has a hyperbolic analogue. 
A perhaps naive first try which comes to mind is what are the best constants
$B^\downarrow_1(p),\,B^\downarrow_2(p)$ so that
\begin{equation}
B^\downarrow_1(p)\,\|f\|_{L^p_p(\C_+)}\le
\|\beur^\downarrow[f]\|_{L^p_p(\C_+)}\le 
B^\downarrow_2(p)\,\|f\|_{L^p_p(\C_+)},\qquad f\in L^p_{p}(\C_+),
\label{eq-tentB}
\end{equation}
It will turn out that 
$$\beur^\downarrow[f]=0,\qquad f\in\text{conj}\,(A^p_p(\C_+)),$$
so that $B^\downarrow_1(p)=0$ necessarily. While $B^\downarrow_2(p)$ exists
boundedly for all $1<p<+\infty$, the exact value appears to be unknown; for
$p=2$ (the Hilbert space case!) the methods of this paper give 
$B^\downarrow_2(2)\le5$. 
If we want a two-sided estimate, we need to compare the norm of 
$\beur^\downarrow[f]$ with the norm of an expression which vanishes on 
$\text{conj}\,(A^p_p(\C_+))$. The expression that will work is
$$f+4\imag\Eop\Mop[f].$$
\medskip

\section{Commutator identities}

\noindent\bf The commutator of differentiation and multiplication. 
\rm The commutator of between
the differentiation operators $\d,\dbar$ and the multiplication operator 
$\Mop^n$ ($n$ is an integer) is given by
\begin{equation}
[\d,\Mop^n]=\d\Mop^n-\Mop^n\d=-\frac{\imag n}{2}\,\Mop^{n-1},\quad
\dbar\Mop^n-\Mop^n\dbar=\frac{\imag n}{2}\,\Mop^{n-1}.
\label{eq-2.2}
\end{equation}
These relations constitute a key step in the proof of the main theorem.
\medskip

\noindent\bf The commutator of hyperbolic Beurling operators. \rm
The various commutators which can be formed using the operators 
$\beur^\uparrow,\bar\beur^\uparrow,\beur^\downarrow,\bar\beur^\downarrow$
can be reduced one of the following three:
$$[\beur^\downarrow,\beur^\uparrow], \quad
[\beur^\downarrow,\bar\beur^\uparrow],\quad 
[\beur^\downarrow,\bar\beur^\downarrow].$$
 We first consider the commutator 
$$[\beur^\downarrow,\beur^\uparrow]=\beur^\downarrow\beur^\uparrow-
\beur^\uparrow\beur^\downarrow.$$
To simplify our work, we introduce the compressed Beurling operator $\beur_+$.
It is given by 
$$\beur_+[f](z)=-\text{pv}\int_{\C_+}\frac{f(w)}{(z-w)^2}\diff A(w),\qquad
z\in\C_+.$$
Moreover, we let $\proj_0$ denote the orthogonal projection $L^2(\C_+)\to
A^2(\C_+)$, which is given explicitly by
$$\proj_0[f](z)=-\int_{\C_+}\frac{f(w)}{(z-\bar w)^2}\diff A(w),
\qquad z\in\C_+.$$ 
In terms of these operators, we have
$$\beur^\uparrow=\beur_{+}-\bar\proj_0,
\quad\beur^\downarrow=\beur_{+}-\proj_0,$$
and so
$$\beur^\downarrow\beur^\uparrow=(\beur_{+}-\proj_0)(\beur_{+}-\bar\proj_0)
=\beur_{+}^2-\beur_{+}\bar\proj_0-\proj_0\beur_{+}+\proj_0\bar\proj_0,$$
while
$$\beur^\uparrow\beur^\downarrow=(\beur_{+}-\bar\proj_0)(\beur_{+}-\proj_0)
=\beur_{+}^2-\beur_{+}\proj_0-\bar\proj_0\beur_{+}+\bar\proj_0\proj_0.$$
Next, $A^2(\C_+)$ and $\text{conj}(A^2(\C_+))$ are orthogonal to one another
in $L^2(\C_+)$, and therefore $\proj_0\bar\proj_0=\bar\proj_0\proj_0=
{\mathbf0}$, where $\mathbf0$ stands for the zero operator. For perhaps
less obvious reasons (cf. \cite{BarHed}), we also have
$\proj_0\beur_{+}=\beur_{+}\bar\proj_0={\mathbf0}$.
So, the commutator simplifies significantly:
$$[\beur^\downarrow,\beur^\uparrow]=\beur^\downarrow\beur^\uparrow-
\beur^\uparrow\beur^\downarrow=\beur_+\proj_0+\bar\proj_0\beur_+.$$
Put ${\Qop}_1=-\d\proj_0$, which is given explicitly by
$$\Qop_1[f](z)=\int_{\C_+}\frac{2f(w)}{(z-\bar w)^3}\diff A(w),
\qquad z\in\C_+.$$ 
A calculation verifies that
$$\bar\proj_0\beur_+=\bar\Qop_1\Mop,\quad \beur_+\proj_0=-\Mop\Qop_1,$$
so that
$$[\beur^\downarrow,\beur^\uparrow]=\bar\Qop_1\Mop-\Mop\Qop_1.$$

The second commutator to be considered is
$$[\beur^\downarrow,\bar\beur^\uparrow]=\beur^\downarrow\bar\beur^\uparrow-
\bar\beur^\uparrow\beur^\downarrow.$$
Following along the lines of the preceding commutator calculation, we find that
$$[\beur^\downarrow,\bar\beur^\uparrow]=[\beur_+,\bar\beur_+]
-\proj_0\bar\beur_+-\beur_+\proj_0=[\beur_+,\bar\beur_+]+\Mop\Qop_1-
\Qop_1\Mop.$$
We readily find that
$$\bar\beur_+\beur_+=\id-\bar\proj_0,\quad 
\beur_+\bar\beur_+=\id-\proj_0,$$
where $\id$ is the identity operator, so that
$$[\beur_+,\bar\beur_+]=\bar\proj_0-\proj_0,$$
and consequently
$$[\beur^\downarrow,\bar\beur^\uparrow]=\bar\proj_0-\proj_0+\Mop\Qop_1
-\Qop_1\Mop.$$
This commutator relation is important because it gives that
\begin{equation*}
\|\bar\beur^\uparrow[f]\|_{L^2(\C_+)}^2-
\|\beur^\downarrow[f]\|_{L^2(\C_+)}^2
=\|\bar\proj_0[f]\|^2_{L^2(\C_+)}
-\|\proj_0[f]\|^2_{L^2(\C_+)}
+2\re\langle\Qop_1[f],\Mop[f]\rangle_{L^2(\C_+)}.
\end{equation*}

The third commutator is 
$$[\beur^\downarrow,\bar\beur^\downarrow]
=\beur^\downarrow\bar\beur^\downarrow-
\bar\beur^\downarrow\beur^\downarrow.$$
The method employed above yields that
$$[\beur^\downarrow,\bar\beur^\downarrow]=\bar\proj_0-\proj_0
+\Mop\Qop_1-\Mop\bar\Qop_1,$$
and, as a consequence, we may also derive that
$$[\beur^\uparrow,\bar\beur^\uparrow]=\bar\proj_0-\proj_0
+\bar\Qop_1\Mop-\Qop_1\Mop.$$

\section{A hyperbolic Liouville-type theorem}

\noindent\bf A Liouville-type theorem for the hyperbolic plane. \rm 
By Liouville's theorem, the only bounded harmonic functions in the 
complex plane are the constants. If we ask the functions to be in $L_p(\C)$
as well, the only harmonic function is the constant $0$. 
A hyperbolic plane analogue of this statement is offered by the following
(see \cite{HedPar} for details). 

\begin{thm} {\rm (Hedenmalm-Parissis)}
Suppose a function $f\in L^p_q(\C_+)$ is harmonic in $\C_+$, where $q$ is 
real while $1\le p<+\infty$. If $q\le-1$, then $f=0$. 
On the other hand, if $-1<q$, there are nontrivial harmonic functions $f$ in 
$L^p_q(\C_+)$. 
\label{thm-L}
\end{thm}

\begin{rem} (a)
For $0<p<1$ and $q\le-2$ the theorem follows from N. Suzuki 
\cite{Suz}. 
\end{rem}

\noindent\bf A biharmonic Liouville-type theorem for the hyperbolic plane. \rm 
A function $f$ with $\bDelta^2f=0$ is said to be {\em biharmonic}.  
An example of such a function is $f=f_1+\Mop f_2$, where $f_1,f_2$ are both
harmonic.

\begin{thm} {\rm (Hedenmalm-Parissis)}
Suppose a function $f\in L^p_q(\C_+)$ is biharmonic in $\C_+$, where 
$-\infty<q\le-1$ while $1\le p<+\infty$. 
Then $\Mop^{-1}f$ is harmonic in $\C_+$. 
\label{thm-L-bih}
\end{thm}

\section{Main results}

\noindent\bf The norm estimate of the hyperbolic plane Beurling transform. 
\rm We now state our main theorem.

\begin{thm} $(1<p<+\infty)$
The operators 
$$\beur^\downarrow:\,L^p_{\!-\!p}(\C_+)\to L^p_{\!-\!p}(\C_+)$$
and
$$\beur^\uparrow:\,L^p_{p}(\C_+)\to L^p_{p}(\C_+)$$
are bounded. Indeed, we have the norm estimate
\begin{equation*}
B(p)^{-1}\big\|f+4\imag\Eop\Mop[f]\big]\big\|_{L^p_p(\C_+)}
\le\|\beur^\downarrow[f]\|_{L^p_{p}(\C_+)}
\le B(p)\big\|f+4\imag\Eop\Mop[f]\big]\big\|_{L^p_p(\C_+)}
\end{equation*}
for all $f\in L^p_p(\C_+)$. Here, $B(p)$ denotes the norm of
the Beurling transform on $L^p(\C)$, and the constants are optimal on both 
sides of the estimate. 
\label{thm-main}
\end{thm}


\begin{proof}
As we saw in the previous section, it is a consequence of the Hardy 
inequality \eqref{eq-hardy-2'} that $\Cop^\downarrow[f]$ and 
$\bar\Cop^\downarrow[f]$ are in $L^p(\C_+)$ provided that 
$f\in L^p_p(\C_+)$.

Next, we shall assume $f\in L^p_p(\C_+)$ is of the form $f=\bDelta F$,
for some $F\in C^\infty_0(\C_+)$. We first claim that the collection of such
$f$ is dense in $L^p_p(\C_+)$. To this end, suppose 
$g\in L^{p'}(\C_+)$ ($p'=p/(p-1)$ is the dual exponent) is such that
$$\int_{\C_+}(\im z)\bDelta F(z)\,\bar g(z)\,\diff A(z)=0.$$
By Green's formula, we get, in the sense of distribution theory,
$$\int_{\C_+}F(z)\,\bDelta((\im z)\bar g(z))\,\diff A(z)=0,$$
for all $F\in C^\infty_0(\C_+)$. It follows that 
$$\bDelta\Mop[g]=0,$$
so that $\Mop[g]\in L^{p'}(\Hplane)$ is harmonic. By Theorem \ref{thm-L},
$g=0$, and the claim follows.

In terms of the function $F$, we have 
$$\Cop^\downarrow[f]=\d F,\qquad \bar\Cop^\downarrow[f]=\dbar F.$$
We now see that the norm estimate of the theorem follows once it has been 
established that
\begin{multline}
B(p)^{-1}\big\|\Mop\bDelta F+\tfrac{\imag}2\,(\d F+\dbar F)\big\|_{L^p(\C_+)}
\\
\le\|\Mop\d^2F\|_{L^p(\C_+)}
\le B(p)\big\|\Mop\bDelta F+\tfrac{\imag}2\,(\d F+\dbar F)\big\|_{L^p(\C_+)}.
\label{eq-1.1.1}
\end{multline}
To this end, we introduce the auxiliary function
$$G=\Mop^2\d\Mop^{-1}[F]\in C^\infty_0(\C_+).$$
We may think of $C^\infty_0(\C_+)$ as a subspace of $C^\infty_0(\C)$ by 
extending the functions to vanish where they were previously undefined.
In particular, it follows from \eqref{eq-beur2} that
\begin{equation}
\frac{1}{B(p)}\|\dbar G\|_{L^p(\C_+)}\le
\|\d G\|_{L^p(\C_+)}\le B(p)\|\dbar G\|_{L^p(\C)_+}.
\label{eq-beur3}
\end{equation}
We first calculate $\d G$, using \eqref{eq-2.2}:
\begin{multline*}
\d G=\d\Mop^2\d\Mop^{-1}[F]=(\Mop^2\d-\imag\Mop)\d\Mop^{-1}[F]=
(\Mop^2\d-\imag\Mop)(\Mop^{-1}\d+\tfrac{\imag}{2}\Mop^{-2})[F]\\
=\Mop^2\d\Mop^{-1}\d F+\tfrac{\imag}{2}\Mop^2\d\Mop^{-2}[F]-\imag\d
+\tfrac12\Mop^{-1}[F]\\
=\Mop^2(\tfrac{\imag}{2}\Mop^{-2}+\Mop^{-1}\d)\d F
+\tfrac{\imag}{2}\Mop^2(\imag\Mop^{-3}+\Mop^{-2}\d)F-\imag\d F
+\tfrac12\Mop^{-1}[F]\\
=\tfrac{\imag}{2}\d F+\Mop\d^2 F-\tfrac12\Mop^{-1}[F]+
\tfrac{\imag}{2}\d F-\imag\d F+\tfrac12\Mop^{-1}[F]=\Mop\d^2 F.
\end{multline*}
We next calculate $\dbar G$ in the same manner:
\begin{multline*}
\dbar G=\dbar\Mop^2\d\Mop^{-1}[F]=(\Mop^2\dbar+\imag\Mop)\d\Mop^{-1}[F]=
(\Mop^2\dbar+\imag\Mop)(\Mop^{-1}\d +\tfrac{\imag}{2}\Mop^{-2})[F]\\
=\Mop^2\dbar\Mop^{-1}\d F+\tfrac{\imag}{2}\Mop^2\dbar\Mop^{-2}[F]+\imag\d F
-\tfrac12\Mop^{-1}[F]\\
=\Mop^2(-\tfrac{\imag}{2}\Mop^{-2}+\Mop^{-1}\dbar)\d F
+\tfrac{\imag}{2}\Mop^2(-\imag\Mop^{-3}+\Mop^{-2}\dbar) F+\imag\d F
-\tfrac12\Mop^{-1}[F]\\
=-\tfrac{\imag}{2}\d F+\Mop\bDelta F+\tfrac12\Mop^{-1}[F]+
\tfrac{\imag}{2}\dbar F+\imag\d F
-\tfrac12\Mop^{-1}[F]=\Mop\bDelta F+\tfrac{\imag}{2}(\d F+\dbar F).
\end{multline*}
The claimed estimate \eqref{eq-1.1.1} is now an immediate consequence 
of \eqref{eq-beur3}. The sharpness of the constants is discussed in the last
section. 
Except for that point, the proof is complete.
\end{proof}

\section{Analysis of two operators}

\noindent\bf The operators. \rm 
In the context of Theorem \ref{thm-main}, with $p=2$, 
we would like to study the operators
$$\beur^\downarrow:\,L^2_2(\C_+)\to L^2_2(\C_+)$$
and 
$$\id+4\imag\Eop\Mop:
\,L^2_2(\C_+)\to L^2_2(\C_+)$$
with respect to range and null space ($\id$ is the identity operator).
It is a curious fact that this problem -- for the second operator -- 
is intimately connected with the
classical Whittaker (or Kummer) ordinary differential equation (see, e. g.,
\cite{AS}, \cite{MOS}, or Wolfram MathWorld, Wikipedia). In view of 
Theorem \ref{thm-main}, the null spaces of the two operators coincide,
which is why we only characterize the null space of the second operator.
\medskip
 
\noindent\bf The range of the operator $\beur^\downarrow$. \rm
The range of $\Mop\beur^\downarrow$ is a subspace of $L^2(\C_+)$, and
studying the range of $\Mop\beur^\downarrow$ is equivalent to studing
the range of $\beur^\downarrow$.
Let $h\in L^2_{\!-\!2}(\C_+)$ be such that $\Mop^{-1}[h]\in L^2(\C_+)$ 
is perpendicular to the range of $\Mop\beur^\downarrow$.
From the proof of
Theorem \ref{thm-main}, we see that this is the same as requiring that
$$\big\langle \Mop^{-1}[h],\Mop\d^2 F
\big\rangle_{L^2(\C_+)}=0,\qquad F\in C^{\infty}_0(\C_+).$$
By dualizing we see that this is the same as
$$\big\langle \dbar^2 h,F
\big\rangle_{L^2(\C_+)}=0,\qquad F\in C^{\infty}_0(\C_+),$$
that is,
$$\dbar^2 h=0.$$
This means that $h$ is bi-analytic in $\C_+$, and hence of the form
$h=h_1+\Mop h_2$, where $h_1,h_2$ are analytic in $\C_+$. In particular,
$h$ is biharmonic, and since $h\in  L^2_{\!-\!2}(\C_+)$,  
Theorem \ref{thm-L-bih} gives that $h_1=0$, so that $h=\Mop  h_2$ where
$h_2\in A^2(\C_+)$. We conclude that the $L^2(\C_+)$-closure of the range of 
$\Mop\beur^\downarrow$ equals
$$L^2(\C_+)\ominus A^2(\C_+).$$

\medskip

\noindent\bf The range of the operator $\id+4\imag\Eop\Mop$. \rm
Let $h\in L^2_{\!-\!2}(\C_+)$ be such that $\Mop^{-1}[h]\in L^2(\C_+)$ 
is perpendicular to the range of the above operator. From the proof of
Theorem \ref{thm-main}, we see that this is the same as requiring that
$$\big\langle \Mop^{-1}[h],\Mop\bDelta F+\tfrac{\imag}{2}(\d+\dbar)F
\big\rangle_{L^2(\C_+)}=0,\qquad F\in C^{\infty}_0(\C_+).$$
By dualizing we see that this is the same as
$$\big\langle \bDelta h+\tfrac{\imag}{2}(\d+\dbar)\Mop^{-1}[h],F
\big\rangle_{L^2(\C_+)}=0,\qquad F\in C^{\infty}_0(\C_+),$$
that is,
$$\bDelta h+\frac{\imag}{2}(\d+\dbar)\Mop^{-1}[h]=0.$$
Since
$$\d+\dbar=\frac{\partial}{\partial x},$$
this amounts to the differential equation
\begin{equation}
y\bigg(\frac{\partial^2}{\partial x^2}
+\frac{\partial^2}{\partial y^2}\bigg)h
+2\imag\frac{\partial}{\partial x}h=0.
\label{eq-a.1}
\end{equation}

\begin{lem}
A function $h\in L^2_{\!-\!2}(\C_+)$ solves the partial differential equation
\eqref{eq-a.1} in $\C_+$ if and only if 
$\Mop^{-1}[h]\in\text{\rm conj}(A^2(\C_+))$. 
\label{lem-a1}
\end{lem}

\begin{proof}
Let
$${\widehat h}(\xi,y)=\int_{-\infty}^{+\infty}\e^{-\imag x\xi}h(x,y)\,
\diff x$$
denote the partial Fourier transform with respect to the $x$ variable.
An application of the partial Fourier transform to the differential equation 
\eqref{eq-a.1} yields
\begin{equation*}
y\bigg(-\xi^2+\frac{\partial^2}{\partial y^2}
\bigg){\widehat h}(\xi,y)-2\xi{\widehat h}(\xi,y)=0,
\end{equation*}
that is,
\begin{equation}
\frac{\partial^2}{\partial y^2}{\widehat h}(\xi,y)
-\bigg(\xi^2+\frac{2\xi}{y}\bigg){\widehat h}(\xi,y)=0.
\label{eq-a.2}
\end{equation}
Next, we put
$$H(\xi,t)={\widehat h}\bigg(\xi,\frac{t}{2|\xi|}\bigg),$$
and see that \eqref{eq-a.2} becomes
\begin{equation}
\frac{\partial^2}{\partial t^2}H(\xi,t)-
\bigg(\frac14+\frac{\sgn(\xi)}{t}\bigg)H(\xi,t)=0.
\label{eq-a.3}
\end{equation}
The requirement that $h\in L^2_{\!-\!2}(\C_+)$ amounts to
\begin{equation}
\int_0^{+\infty}\int_{-\infty}^{+\infty}|H(\xi,t)|^2\,
\frac{|\xi|\diff\xi\diff t}{t^2}<+\infty.
\label{eq-a.4}
\end{equation}
The differential equation \eqref{eq-a.3} is of Whittaker type.
It is well-known that the general solution to the ordinary differential 
equation
$$\frac{d^2}{\partial t^2}X(t)-
\bigg(\frac14+\frac{1}{t}\bigg)X(t)=0$$
is of the form 
$$X(t)=A_1t\e^{t/2}+B_1\,t\e^{-t/2}
\int_0^{+\infty}\e^{-t\theta}\frac{\theta}{1+\theta}\,\diff\theta,$$
where $A_1,B_1$ are constants, while the general solution to the 
ordinary differential equation
$$\frac{d^2}{\partial t^2}Y(t)-
\bigg(\frac14-\frac{1}{t}\bigg)Y(t)=0$$
is of the form 
$$Y(t)=A_2\e^{-t/2}\bigg(1-t\log t-t\int_0^t
\frac{\e^\theta-1-\theta}{\theta^2}\diff\theta\bigg)
+B_2\,t\e^{-t/2},$$
where $A_2,B_2$ are constants. It follows that $H(\xi,t)$ must have form
$$H(\xi,t)=A_1(\xi)t\e^{t/2}+B_1(\xi)\,t\e^{-t/2}
\int_0^{+\infty}\e^{-t\theta}\frac{\theta}{1+\theta}\,\diff\theta,\qquad
\xi>0,$$
and 
$$H(\xi,t)=A_2(\xi)\e^{-t/2}\bigg(1-t\log t-t\int_0^t
\frac{\e^\theta-1-\theta}{\theta^2}\diff\theta\bigg)
+B_2(\xi)\,t\e^{-t/2},\qquad \xi<0.$$
A careful analysis of the behavior of these solutions as $t\to0^+$ and
$t\to+\infty$ shows that \eqref{eq-a.4} is impossible unless $A_1(\xi)=
B_1(\xi)=A_2(\xi)=0$, in which case 
$$H(\xi,t)=0,\qquad \xi>0,$$
and 
$$H(\xi,t)=B_2(\xi)\,t\e^{-t/2},\qquad \xi<0.$$
The function $B_2(\xi)$ must then satisfy
$$\int_{-\infty}^{0}|\xi|\,|B_2(\xi)|^2\diff\xi<+\infty,$$
and the partial Fourier transform $\widehat h$ takes the form
$$\widehat h(\xi,y)=2|\xi|yB_2(\xi)\,\e^{y\xi}\,1_{]-\infty,0]}(\xi).$$
This form of $\widehat h$ is equivalent to the assertion that
$\Mop^{-1}[h]\in\text{conj}(A^2(\C_+))$. 
\end{proof}

We now obtain the closure of the range of the operator. 

\begin{prop}
The closure of the range of
$$\Mop+4\imag\Mop\Eop\Mop:
\,L^2_2(\C_+)\to L^2(\C_+)$$ 
equals $L^2(\C_+)\ominus \text{\rm conj}(A^2(\C_+))$.
\end{prop}

\begin{rem}
One can show that the range of the operator is closed.
\end{rem}
\medskip

\noindent\bf The null space of the operator. 
\rm We turn to the study of the kernel of the operator. So, given
$f\in L^2_2(\C_+)$, we want to know what the solutions to 
\begin{equation}
\Mop[f]+4\imag\Mop\Eop\Mop[f]=0
\label{eq-a.5}
\end{equation}
look like. Let $F\in L^2_{\!-\!2}(\C_+)$ be the associated function
$$F=\bar\Cop^{\uparrow}\Cop^\downarrow[f]
=\Cop^{\uparrow}\bar\Cop^\downarrow[f].$$
Then $\bDelta F=f$, and \eqref{eq-a.5} takes the form
\begin{equation}
\Mop[\bDelta F]+\frac{\imag}{2}[\d F+\dbar F]=0.
\label{eq-a.6}
\end{equation}
By Lemma \ref{lem-a1}, we find that this happens if and only if
$\Mop^{-1}[F]\in \text{conj}(A^2(\C_+))$. From this, we quickly derive
the following characterization of the null space.

\begin{prop}
The null space of the operator 
$$\id+4\imag\Eop\Mop:\,L^2_2(\C_+)\to L^2(\C_+)$$
equals $\text{\rm conj}(A^2_2(\C_+))$. 
\end{prop}
\medskip

\section{Geometric interpretation of the main theorem}

\noindent\bf The hyperbolic plane and differential operators. \rm 
Associated with the half-plane model of the hyperbolic plane $\Hplane$, 
we have the geometrically induced differential operators 
$\d^\uparrow,\dbar^\uparrow$:
$$\d^\uparrow=\Mop\d,\quad \dbar^\uparrow=
\Mop\dbar.$$
After all, the length scale on $\C_+$ should be modified to correspond to 
that of the hyperbolic plane. 
There are also the ``dual'' geometrically induced differential operators
$\d^\downarrow,\dbar^\downarrow$:
$$\d^\downarrow=\Mop^2\d\Mop^{-1},\quad 
\dbar^\downarrow=\Mop^2\dbar\Mop^{-1},$$
with the properties that
$$\langle \d^\uparrow f,g\rangle_{L^2(\C_+)}=
-\langle f,\dbar^\downarrow g\rangle_{L^2(\C_+)},\quad
\langle \dbar^\uparrow f,g\rangle_{L^2(\C_+)}=
-\langle f,\d^\downarrow g\rangle_{L^2(\C_+)},$$
provided at least one of $f,g$ is in the class $C^\infty_0(\C_+)$ of compactly
supported test functions, and the 
other is, say, locally integrable on $\C_+$ (the partial derivatives are 
interpreted in the sense of distribution theory when necessary). 
The hyperbolic Laplacian $\bDelta_\Hplane$ 
is obtained as a combination of two such geometric differential operators:
$$\bDelta_\Hplane=\dbar^\downarrow\d^\uparrow=\d^\downarrow\dbar^\uparrow
=\Mop^2\bDelta.$$
\medskip

\noindent\bf Hyperbolic plane Beurling operators. \rm In analogy with the 
planar Beurling transform $\beur=\d\dbar^{-1}$, we suggest for the hyperbolic
plane $\d^\downarrow(\dbar^\downarrow)^{-1}$ as  the 
``hyperbolic plane Beurling transform''.
Like in the case of the Euclidean plane, there is the matter of the choice of 
$(\dbar^\downarrow)^{-1}$. 
In contrast with the $\dbar$-problem in the plane, given a function 
$f\in L^2(\Hplane):=L^2_{\!-\!2}(\C_+)$, there always exists
a solution $u\in L^2(\Hplane)$ with $\dbar^\downarrow u=f$, and
$$\|u\|_{L^2(\Hplane)}\le 4\|f\|_{L^2(\Hplane)}.$$
This follows from the well-known Hardy inequality in a manner which will be 
explained in a later section.
In particular, there always exists a unique solution $u=u_f$ of minimal norm in
$L^2(\Hplane)$. We write $u_f=[\dbar^\downarrow]^{-1}_{\text{min}}f$
for this minimal solution, and have thus defined the operator 
$[\dbar^\downarrow]^{-1}_{\text{min}}$. In a similar manner, we may define
the operator $[\d^\downarrow]^{-1}_{\text{min}}$. 
It is easy to see that 
$$[\dbar^\downarrow]^{-1}_{\text{min}}=\Mop\Cop^\downarrow\Mop^{-2},\quad
[\d^\downarrow]^{-1}_{\text{min}}=\Mop\bar\Cop^\downarrow\Mop^{-2},$$
and hence 
$[\d^\downarrow]^{-1}_{\text{min}},[\dbar^\downarrow]^{-1}_{\text{min}}$ 
act boundedly on $L^p(\Hplane):=L^p_{\!-\!p}(\C_+)$ for $1<p<+\infty$. 
Moreover, since
$$\d^\downarrow[\dbar^\downarrow]^{-1}_{\text{min}}=
(\Mop^2\d\Mop^{-1})(\Mop\Cop^\downarrow\Mop^{-2})=
\Mop^2\d\Cop^\downarrow\Mop^{-2}=\Mop^2\beur^\downarrow\Mop^{-2},$$
we see that the operator $\beur^\downarrow$ is indeed a Beurling-type operator
associated with the hyperbolic plane.
Theorem \ref{thm-main} may be formulated in these geometric-differential 
operator terms.

\begin{thm} $(1<p<+\infty)$
For $p=2$, we have the norm identity
\begin{equation*}
\big\|\d^\downarrow[\dbar^\downarrow]^{-1}_{\text{\rm min}}f
\big\|_{L^2(\Hplane)}=\big\|f+\tfrac{\imag}{2}
\big([\dbar^\downarrow]^{-1}_{\text{\rm min}}f+
[\d^\downarrow]^{-1}_{\text{\rm min}}f\big)\big\|_{L^2(\Hplane)},
\qquad f\in L^2(\Hplane),
\end{equation*}
while for general $p$, we have
\begin{multline*}
\frac{1}{B(p)}\big\|f+\tfrac{\imag}{2}
\big([\dbar^\downarrow]^{-1}_{\text{\rm min}}f+
[\d^\downarrow]^{-1}_{\text{\rm min}}f\big)\big\|_{L^p(\Hplane)}\le
\big\|\d^\downarrow[\dbar^\downarrow]^{-1}_{\text{\rm min}}f
\big\|_{L^p(\Hplane)}\\
\le B(p)\big\|f+\tfrac{\imag}{2}
\big([\dbar^\downarrow]^{-1}_{\text{\rm min}}f+
[\d^\downarrow]^{-1}_{\text{\rm min}}f\big)\big\|_{L^p(\Hplane)},
\qquad f\in L^p(\Hplane). 
\end{multline*}
The constants are sharp.
\label{thm-main:2}
\end{thm}

\medskip

\noindent\bf Explanation of the sharpness of the constants. \rm
Let introduce, for real $\alpha\ge0$, the half-plane
$$\C_+^\alpha=\big\{z\in\C:\,\im z>-\alpha\big\},$$
supplied with the metric and area measure 
$$\diff s_{\alpha}(z)=\frac{(1+\alpha)|\diff z|}{\alpha+\im z},\quad
\diff A_\alpha(z)=\frac{(1+\alpha)^2\diff A(z)}{(\alpha+\im z)^2},$$
and write 
$$\Hplane^\alpha=\langle\C_+^\alpha,\diff s_{\alpha}\rangle$$
for this model of the hyperbolic plane. For $\alpha=0$ we get the standard
model of the hyperbolic plane, while as $\alpha\to+\infty$ the model flattens
out to give the Euclidean plane in the limit.
We consider the associated multiplication operator
$$\Mop_\alpha[f](z)=\frac{\im z+\alpha}{1+\alpha}\,f(z),\qquad
z\in\C_+^\alpha,$$
and the differential operators
$$\d^\uparrow_\alpha=\Mop_\alpha\d,\quad \dbar^\uparrow=
\Mop_\alpha\dbar,\quad\d^\downarrow=\Mop^2_\alpha\d\Mop^{-1}_\alpha,\quad 
\dbar^\downarrow=\Mop^2_\alpha\dbar\Mop^{-1}_\alpha.$$
The estimate of Theorem \ref{thm-main:2} for general $1<p<+\infty$ now
takes the form (with obvious notation) 
\begin{multline*}
\frac{1}{B(p)}\big\|f+\tfrac{\imag}{2(1+\alpha)}
\big([\dbar^\downarrow_\alpha]^{-1}_{\text{\rm min}}f+
[\d^\downarrow_\alpha]^{-1}_{\text{\rm min}}f\big)
\big\|_{L^p(\Hplane^\alpha)}\le
\big\|\d^\downarrow_\alpha[\dbar^\downarrow_\alpha]^{-1}_{\text{\rm min}}f
\big\|_{L^p(\Hplane^\alpha)}\\
\le B(p)\big\|f+\tfrac{\imag}{2(1+\alpha)}
\big([\dbar^\downarrow_\alpha]^{-1}_{\text{\rm min}}f+
[\d^\downarrow_\alpha]^{-1}_{\text{\rm min}}f\big)\big\|_{L^p(\Hplane^\alpha)},
\qquad f\in L^p(\Hplane^\alpha). 
\end{multline*}
As $\alpha\to+\infty$, the geometry becomes Euclidean, and the estimate becomes
\begin{equation*}
\frac{1}{B(p)}\|f\|_{L^p(\C)}\le
\big\|\d\dbar^{-1}f
\big\|_{L^p(\C)}
\le B(p)\|f\|_{L^p(\C)},
\qquad f\in L^p(\C), 
\end{equation*}
which we recognize as \eqref{eq-beur1}. For this reason, the constants
cannot be improved.
\medskip

\noindent\bf Acknowledgements. \rm The author thanks Ioannis Parissis for
extensive discussions, and Michael Benedicks for his interest in this work.


\end{document}